\documentclass{amsart}

\makeatletter
\providecommand{\@secnumpunct}{. }
\makeatother

\usepackage{graphicx} 
\usepackage{scrextend}
\usepackage{subcaption}
\usepackage{float}
\usepackage[utf8]{inputenc}
\usepackage{mathtools}
\usepackage{placeins}
\usepackage{lipsum}
\usepackage{verbatim}
\usepackage{tikz-cd, tikz}
\usepackage{amsmath}
\usepackage{amsfonts}
\usepackage{nccmath}
\usepackage{amsthm}
\usepackage{amssymb}
\usepackage{setspace}
\usepackage{pifont}
\usepackage{array}
\usepackage{enumitem}
\usepackage[english]{babel}
\usepackage[all]{xy}
\usepackage[autostyle]{csquotes}
\usepackage{aliascnt}    
\usepackage[
  colorlinks=true,
  urlcolor=black,
  linkcolor=blue,
  citecolor=blue, 
  linktocpage=true
]{hyperref}
\usepackage{geometry}
\usepackage[labelformat=simple]{caption}

\usetikzlibrary{plotmarks, positioning, fit}
\usetikzlibrary{calc}

\numberwithin{thmcounter}{section}     
\newaliascnt{thmauto}{thmcounter}

\newaliascnt{Defauto}{thmcounter}

\newaliascnt{lemauto}{thmcounter}

\newaliascnt{propauto}{thmcounter}

\newaliascnt{corauto}{thmcounter}

\newaliascnt{remauto}{thmcounter}

\newaliascnt{notauto}{thmcounter}
\newaliascnt{conjauto}{thmcounter}

\newaliascnt{conauto}{thmcounter}

\newaliascnt{obsauto}{thmcounter}

\newaliascnt{exauto}{thmcounter}

\addto\extrasenglish{}
\addto\extrasenglish{}
\addto\extrasenglish{}

\newcommand{\suchthat}{\ifnum\currentgrouptype=16\middle\fi|}

\makeatletter
\newcommand{\oset}[2]{%
  {\mathop{#2}\limits^{\vbox to 15\ex@{\kern-\tw@\ex@
   \hbox{\scriptsize #1}\vss}}}}
\makeatother

\newtheorem*{theorem*}{Theorem}

\newtheorem{atheorem}{Theorem}

\newtheorem{theorem}[thmauto]{Theorem}
\newtheorem{lemma}[lemauto]{Lemma}
\newtheorem{proposition}[propauto]{Proposition}
\newtheorem{corollary}[corauto]{Corollary}
\newtheorem{conjecture}[conjauto]{Conjecture}

\theoremstyle{definition}
\newtheorem{definition}[Defauto]{Definition}

\theoremstyle{remark}
\newtheorem{remark}[remauto]{Remark}

\DeclareMathOperator{\SL}{SL}
\DeclareMathOperator{\GL}{GL}

\DeclareMathOperator{\C}{C}

\DeclareMathOperator{\St}{St}
\DeclareMathOperator{\redhom}{\widetilde{H}}
\DeclareMathOperator{\rk}{rank}

\DeclareMathOperator{\homology}{H}

\DeclareMathOperator{\stab}{Stab}

\DeclareMathOperator{\Img}{im}

\providecommand{\longrightarrow}{\relbar\joinrel\rightarrow}

\providecommand{\F}{\mathbb{F}}
\providecommand{\Z}{\mathbb Z}
\providecommand{\Q}{\mathbb Q}
\providecommand{\T}{\mathcal T}

\newcounter{cases}
\newcounter{subcases}[cases]
\newenvironment{mycases}
  {%
    \setcounter{cases}{0}%
    \setcounter{subcases}{0}%
    \def\case
      {%
        \par\noindent
        \refstepcounter{cases}%
        \textbf{Case \thecases.}
      }%
  }
  {%
  }
\renewcommand*\thecases{\arabic{cases}}

\numberwithin{equation}{section}

\usepackage{titlesec}
\titlespacing*{\section}
{0pt}{2ex plus 1ex minus .2ex}{1ex plus .2ex}
\titlespacing*{\subsection}
{0pt}{2ex plus 1ex minus .2ex}{1ex plus .2ex}
\titlespacing*{\subsubsection}
{0pt}{2ex plus 1ex minus .2ex}{1ex plus .2ex}
\titleformat{\paragraph}[hang]{\normalfont\normalsize\bfseries}{\theparagraph}{1em}{}[]
\titlespacing*{\paragraph}{0pt}{1em}{1em}

\title{Codimensions one and two cohomology of Hecke congruence subgroups}
\thanks{
Jeremy Miller was partially supported by National Science Foundation (NSF) grants DMS-2202943 and DMS2504473 as well as a Simons Foundation Travel Support for Mathematicians grant. Tatiana Abdelnaim was supported by National Science Foundation grant DMS-2405310 as well as a Simons Foundation Travel Support for Mathematicians grant.}
\author{Tatiana Abdelnaim and Jeremy Miller}
\date{}

\begin{document}
 
\begin{abstract} For $n\geq 1$ and $p$ a prime, the Hecke congruence subgroup $\Gamma_{0,n}(p)\leq \SL_n(\Z)$ is the subgroup of matrices whose first column is of the form $(*,0,\dots,0)^t\bmod p$. Borel--Serre showed that $\Gamma_{0,n}(p)$ has virtual cohomological dimension $\binom{n}{2}$. The first author proved that the rational cohomology in this top degree $\binom{n}{2}$ vanishes for $n$ sufficiently large compared to $p$. We prove analogous results in codimension $1$ and $2$. 
\end{abstract}
\maketitle

\section{Introduction}

Let $n\geq 1$ and $p$ a prime. Recall that the \emph{Hecke congruence subgroup} of $\SL_n(\Z)$ is defined as follows: 
\begin{align*}\Gamma_{0,n}(p)&=\left\{g\in\SL_n(\Z)~\middle\vert\ g\equiv \begin{pmatrix*}
* & * & \dots& * \\
0 & * &\dots&*\\
\vdots &\vdots & \ddots&\vdots \\
0 & * & \dots & *
\end{pmatrix*}\bmod p \right\}\\
&=\left\{g\in\SL_n(\Z)~\middle\vert\ g\cdot e_1\equiv \lambda e_1\bmod p~\text{for some $\lambda\in\F_p^\times$}\right\}.\end{align*}
The cohomology of Hecke congruence subgroups are closely related to modular forms for $n=2$ and automorphic forms more generally.

In this paper, we investigate the cohomology in large degrees. Before specializing to Hecke congruence subgroups, we review what is known more generally about the high degree rational cohomology groups of finite index subgroups of $\SL_n(\Z)$.

It follows from work of Borel--Serre \cite{BS73} that, for $\Gamma \leq \SL_n(\Z)$ finite index, $$\homology^{i}\left(\Gamma;\Q\right) \cong 0 \text{ for } i > \binom{n}{2}.$$ For each finite index subgroup $\Gamma$, it is natural to ask what is the largest degree in which $\Gamma$ has nontrivial rational cohomology. The general intuition is that for $\Gamma$ ``large'', one should expect a larger vanishing range than what Borel--Serre prove and for $\Gamma$ ``small,'' Borel--Serre's vanishing result should be sharp. We refer to $\homology^{\binom{n}{2} - i}\left(\Gamma;\Q\right)$ as the codimension-$i$ cohomology of $\Gamma$. At one extreme, we have the Church--Farb--Putman conjecture \cite{CFPconj} which predicts that the codimension-$i$ rational cohomology of $\SL_n(\Z)$ vanishes for all $n$ sufficiently large.

\begin{conjecture}[Church--Farb--Putman]
\noindent  The codimension-$i$ cohomology of $\SL_n(\Z)$ vanishes for $n \geq i+2$.
\end{conjecture}

In other words, they conjectured that $\homology^{i}\left(\SL_n(\Z);\Q\right) \cong 0$ for $i \geq \binom{n}{2} - n+2$. This has been verified for $i=0$ by work of Lee--Szczarba \cite{LS}, $i=1$ by work of Church--Putman \cite{CP}, and $i=2$ by work of Brück, Patzt, Sroka, Wilson and the second author \cite{BMPSW}. 

In contrast, if a finite index subgroup $\Gamma \leq \SL_n(\Z)$ is contained in a principal congruence subgroup, then $$\homology^{\binom{n}{2}}\left(\Gamma;\Q\right) \neq 0$$ by work of Paraschivescu \cite{Paraschivescu}.

The first author \cite[Theorem A]{Abdelnaim} proved that the codimension zero cohomology of $\Gamma_{0,n}(p)$ vanishes for $n\geq \frac{p+14}{6}$. We extend this to codimensions $1$ and $2$.

\begin{atheorem}\label{codim1}
    Let $p$ be a prime. Then
    $\homology^{\binom{n}{2}-1}\left(\Gamma_{0,n}(p);\Q\right) \cong 0$ for $n\geq 4p+5$. 
\end{atheorem}
\begin{atheorem}\label{codim2}
   Let $p$ be a prime. Then
    $\homology^{\binom{n}{2}-2}\left(\Gamma_{0,n}(p);\Q\right) \cong 0$ for $n\geq 7p+8$. 
\end{atheorem}

The first author \cite[Theorem B]{Abdelnaim} proved that   $\homology^{\binom{n}{2}}\left(\Gamma_{0,n}(p);\Q\right)$ does not vanish for small $n$ and $p$. Thus, it seems unlikely that there is a range independent of $p$ for which codimension-$1$ and codimension-$2$ rational cohomology of $\Gamma_{0,n}(p)$ vanishes. On the other hand, we do not know if the above stable ranges are sharp.

\subsection*{Proof strategy.}
Our main technical tool is Borel–Serre duality \cite{BS73}. Borel–Serre proved that if $\Gamma$ is a finite index subgroup of $\SL_n(\Z)$, then \[\homology^{\binom{n}{2}-i}\left(\Gamma;\Q\right)\cong \homology_i\left(\Gamma;\St_n(\Z)\right).\] Here $\St_n(\Z)$ is the Steinberg module (\autoref{def:St}). It is defined as the top homology of the Tits building of $\Z^n$, the simplicial complex whose simplices correspond to flags of proper nonzero  summands. This translates \autoref{codim1} and \autoref{codim2} to questions concerning $\homology_i\left(\Gamma_{0,n}(p);\St_n(\Z)\right)$ for $i=1$ and $2$ respectively. We show vanishing of these low-degree homology groups using a resolution established by Patzt, Putman, and the second author \cite[Theorem D']{MPP} and a spectral sequence argument similar to one used by Nagpal, Patzt, and the second author \cite[Proof of Theorem 7.1]{MNP}.

\subsection*{Acknowledgments.}
We thank Peter Patzt for helpful conversations and comments on an earlier version of this paper.

\section{Construction of a spectral sequence}
We begin by recalling the Steinberg module and a chain complex due to Nagpal, Patzt, and the second author \cite{MNP}. 

\begin{definition}
    Let $A$ be a free abelian group of rank $n$. Define the \emph{Tits building} $\T(A)$ to be the simplicial complex whose vertices are proper nonzero summands of $A$, and whose $k$-simplices are given by flags \[0\subsetneq V_0\subsetneq\cdots \subsetneq V_k\subsetneq A.\] 
\end{definition}
\begin{theorem}[{\cite{solomon-tits}}]
    Let $A$ be a free abelian group of rank $n$. $\T(A)$ is homotopy equivalent to a wedge of $(n-2)$-spheres. 
\end{theorem}
 \begin{definition}\label{def:St}
     The \emph{Steinberg module} $\St(A)$ is defined as \[\St(A):=\redhom_{n-2}(\T(A);\Q).\] We write $\St_n(\Z):=\St(\Z^n)$.
 \end{definition}

\begin{theorem}[{\cite{Bykovski}}] \label{Bykovski}
 Let $A$ be a free abelian group of rank $n$. The Steinberg module $\St(A)$ is generated by symbols $[v_1,\dots,v_n]$, one for each ordered basis $\{v_1,\dots,v_n\}$ of $A$, subject to:\begin{enumerate}
    \item$[v_1,v_2,v_3,\dots,v_n]=[v_1,v_1+v_2,v_3,\dots,v_n]+[v_1+v_2,v_2,v_3,\dots,v_n]$
    \item \label{signByk}$[v_{\sigma(1)},v_{\sigma(2)},\dots,v_{\sigma(n)}]=(-1)^\sigma[v_1,v_2,\dots,v_n]$ for every $\sigma\in \Sigma_n$
    \item $[v_1,v_2,\dots,v_n]=[-v_1,v_2,\dots,v_n].$
\end{enumerate}
The generators $[v_1,\dots,v_n]$ are called apartments.
\end{theorem}
 $\GL(A)$ acts on $\St(A)$ by $g\cdot[v_1,\dots,v_n]=[g\cdot v_1,\dots,g\cdot v_n]$.
The proofs of \autoref{codim1} and \autoref{codim2} are based on a spectral sequence arising from a chain complex constructed by Nagpal, Patzt, and the second author \cite{MNP}.

  \begin{theorem}[{\cite{MNP}}]
  Let $X$ and $Y$ be free abelian groups of rank $n$ and $m$ respectively. There is a multiplication map \[\mu_{X,Y}:\St(X)\otimes\St(Y)\rightarrow \St(X\oplus Y)\]given by apartment concatenation, \[[v_1,\dots,v_n]\otimes[v_1',\dots,v_n']\mapsto [v_1,\dots,v_n,v_1',\dots,v_n'].\]
  \end{theorem}
  For the rest of the paper, let \[C_k(n):=\bigoplus\limits_{A_0\oplus\dots\oplus A_k=\Z^n}\St(A_0)\otimes\dots\otimes\St(A_k).\]
 Here, and throughout the rest of the paper, all decompositions $A_0\oplus\dots\oplus A_k=\Z^n$ are ordered and have nonzero summands. Note that $C_0(n)=\St_n(\Z)$.  
 
 \begin{theorem}[{\cite{MNP}}]
For each $n$, there is a chain complex of $\GL_n(\Z)$-modules
\[
\dots \longrightarrow C_2(n)
\longrightarrow C_1(n)\longrightarrow
C_0(n)\longrightarrow 0,
\]
whose differentials are given by the alternating sum of maps induced by $$\mu_{A_i,A_{i+1}}:\St(A_i) \otimes \St(A_{i+1}) \to \St(A_i \oplus A_{i+1}). $$
\end{theorem}

The following theorem was proved by Patzt, Putman, and the second author \cite{MPP}, building on earlier work  of Lee--Szczarba \cite{LS}, Church--Putman \cite{CP}, second author with Br\"uck, Patzt, Sroka, and Wilson \cite{BMPSW}, and the second author with Patzt, and Wilson \cite{MPW}.

 \begin{theorem}[{\cite[Theorem D']{MPP}}]\label{exact}
   For $n\geq 4$, $\homology_i(C_*(n)) \cong 0$ for $i \leq 2$.

 \end{theorem}

 We next introduce the group that will serve as the primary object in our computations.
\begin{definition}
For $p$ a prime, let $A$ be a free abelian group and let $v\in A$. Define
\[G_A^p(v):=\{g\in\SL(A)\mid g\cdot v\equiv v \bmod pA\}.
\]

When $A=\Z^n$, we write $G_n^p(v):=G_{\Z^n}^p(v).$ If $v=e_1$, we abbreviate $G_{A}^p:=G_{A}^p(e_1).$
\end{definition}

By definition, if $v\equiv 0 \bmod pA$, then $G_A^p(v)=\SL(A).$

The group $G_n^p$ is a finite index subgroup of $\Gamma_{0,n}(p)$ that is easier to study. We will first establish vanishing results for the homology of $G_n^p$ and then use them to deduce corresponding vanishing results for $\Gamma_{0,n}(p)$.

The action of $G_n^p$ on $\C_*(n)$ gives rise to a spectral sequence. Here and throughout the rest of the paper, we set \[S_k=\{(A_0,\ldots,A_k)\mid A_0\oplus\dots\oplus A_k=\Z^n\},\]and for $\vec{A}=(A_0,\dots,A_k)\in S_k$, we denote by $[\vec{A}]$ the class of $\vec{A}$ in $G_n^p\backslash S_k$. For each $(A_0,\dots,A_k)\in S_k$, we fix elements $v_i\in A_i$ such that \[e_1\equiv v_0+\dots+v_k\bmod p,\]and use this notation throughout.
\begin{lemma}\label{lem:ss}
The $E^1$-page of the hyperhomology spectral sequence
\[E^1_{kh}(n)=\homology_h(G_n^p;C_k(n))
\Longrightarrow
\homology_{k+h}(G_n^p;C_*(n))\]
has the following alternative description: \[E^1_{kh}(n)\cong
\bigoplus_{[\vec{A}]=[A_0,\dots,A_k]\in G_n^p\backslash S_k}\homology_h\left(
\stab_{G_n^p}(\vec{A});\St(A_0)\otimes\dots\otimes\St(A_k)
\right).\]
\end{lemma}
\begin{proof}
This is the spectral sequence obtained from the filtered double complex
computing $\homology_*(G_n^p;C_*(n))$. The group $G_n^p$ acts on $S_k$ by
\[g\cdot (A_0,\ldots,A_k)=(g\cdot A_0,\dots,g\cdot A_k).\]
This action induces an action on $C_k(n)$. Decomposing $S_k$ into $G_n^p$-orbits, we obtain
\[
C_k(n)\cong
\bigoplus_{[\vec{A}]=[A_0,\dots,A_k]\in G_n^p\backslash S_k}\operatorname{Ind}_{\stab_{G_n^p}(\vec{A})}^{G_n^p}
\left(\St(A_0)\otimes\dots\otimes\St(A_k)\right).\]
Therefore, by Shapiro's lemma,
\[\homology_h(G_n^p;C_k(n))
\cong
\bigoplus_{[\vec{A}]\in G_n^p\backslash S_k}
\homology_h\left(
\stab_{G_n^p}(\vec{A});\St(A_0)\otimes\dots\otimes\St(A_k)
\right).\qedhere\]
\end{proof}
\begin{corollary}\label{E-infty}
   For $n\geq 4$,  the spectral sequence in \autoref{lem:ss} satisfies \[E^\infty_{01}(n)=E^\infty_{02}(n)=0.\]
\end{corollary}
\begin{proof}
By \autoref{exact}, the spectral sequence in \autoref{lem:ss} converges to zero in total degrees $0,1,$ and $2$. In particular, for $n\geq 4$,\[E^\infty_{01}(n)=E^\infty_{02}(n)=0.\qedhere\]
\end{proof}
Observe that \[E^1_{01}(n)=\homology_1(G_n^p;\St_n(\Z)) \text{\,\, and \,\,} E^1_{02}(n)=\homology_2(G_n^p;\St_n(\Z)).\]
On the other hand,  \[E^\infty_{01}(n)=0 \text{\,\, and \,\,} E^\infty_{02}(n)=0\] for $n \geq 4$. Thus, we can establish \autoref{codim1} and \autoref{codim2} if we can show 
$E^1_{01}(n) \cong E^\infty_{01}(n)$ and $E^1_{02}(n) \cong E^\infty_{02}(n)$ for large $n$. There are no nontrivial differentials out of these groups so our task is to rule out nontrivial differentials into these groups. We will first rule out $d^2$ and higher differentials. This will imply that the differentials \[f\colon E^1_{11}(n)\rightarrow E^1_{01}(n)\]and \[g\colon E^1_{12}(n)\rightarrow E^1_{02}(n)\]are surjective. We will then show these differentials are zero.

\begin{figure}[H]
  \centering
\begin{tikzpicture}[scale=1.5]
    \draw[->] (0,0) -- (4,0) node[right] {$k$};
    \draw[->] (0,0) -- (0,3) node[above] {$h$};

    \node[fill=white, inner sep=2pt] at (0.2,0.2) {$*$};
    \node[fill=white, inner sep=2pt] at (1.2,0.2) {$*$};
    \node[fill=white, inner sep=2pt] at (2.2,0.2) {$*$};
    \node[fill=white, inner sep=2pt] at (0.2,1.2) {$*$};
    \node[fill=white, inner sep=2pt] at (1.2,1.2) {$*$};
    \node[fill=white, inner sep=2pt] at (2.2,1.2) {$*$};
    \node[fill=white, inner sep=2pt] at (3.2,0.2) {$*$};
    \node[fill=white, inner sep=2pt] at (1.2,2.2) {$*$};
    \node[fill=white, inner sep=2pt] at (0.2,2.2) {$*$};
    \draw[->] (1.1,1.2) -- node[above] {$f$}(0.3,1.2);
    \draw[->] (1.1,2.2) -- node[above] {$g$}(0.3,2.2);
  \end{tikzpicture}
 \caption{$E^1$-page.}
  \label{fig:1}
  \end{figure}
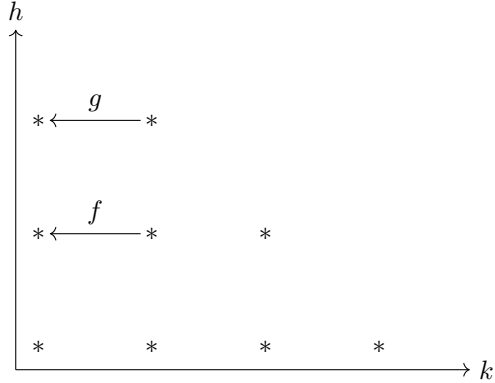

\section{The bottom row of the spectral sequence}
In this section, we study the bottom row of the spectral sequence stated in \autoref{lem:ss}. Our next goal is to prove a vanishing result for the terms \[E^1_{k0}(n)\cong\bigoplus\limits_{[\vec A]=[(A_0,\dots,A_k)]\in G_n^p/S_k}\homology_0\left(\stab_{G_n^p}(\vec{A});\St(A_0)\otimes\dots\otimes\St(A_k)\right).\]

        \begin{proposition}\label{coinv K=0}
            Let $p$ be a prime and $A$ be a free abelian group of rank $n$ with $n\geq p+2$. Let $v\in A$ be such that $v\not\equiv 0\bmod pA$. Then  \[\homology_0\left(G_A^p(v);\St(A)\right)\cong 0.\]
        \end{proposition}
        \begin{proof}
        Fix $n\geq p+2$. Pick an isomorphism between $A$ and $\Z^n$ sending $v$ to $e_1$. This identifies $G_A^p(v)$ with $G_n^p$. As $\homology_0\left(G_n^p;\St_n(\Z)\right)\cong \St_n(\Z)_{G_n^p}$, it is thus enough to prove \[\St_n(\Z)_{G_n^p} \cong 0.\] 
        
       \noindent Let $[v_1,\dots,v_{n}]$ be a generator of $\St(A)$. There exists $(a_1,\dots,a_{n})\in\Z^{n}\setminus \{0\}$ satisfying \begin{equation}\label{eq:e1}\begin{aligned}
            e_1&=a_1v_1+a_2v_2+a_3v_3+\dots+a_{n}v_{n}\\
            &=a_1v_1+\dots+a_j( v_j +v_\ell) +\dots+(a_\ell-a_j) v_\ell+\dots+a_{n}v_{n}\\
            &=a_1v_1+\dots+(a_j-a_\ell)v_j +\dots+a_\ell( v_j+v_\ell)+\dots+a_{n}v_{n}.
        \end{aligned}\end{equation}
        Additionally, \begin{equation}\label{eq:v}[v_1,\dots,v_j,\dots,v_{\ell},\dots,v_{n}]=\pm [v_1,\dots,v_j+v_{\ell},\dots,v_{\ell},\dots,v_{n}]\pm [v_1,\dots,v_j,\dots,v_j+v_{\ell},\dots,v_{n}].\end{equation}
        Since $n\geq p+2$, one of the following cases occurs.\begin{mycases}
        \case $a_j \equiv 0 \bmod p$ and $a_r \equiv a_s \bmod p$ for some $1\le j<r<s\leq n$.
        
        In this case, we define $g \in \SL_n(\Z)$ by \begin{align*}
                v_j&\mapsto -v_j, \\
                v_r&\mapsto v_s,\\
                v_s&\mapsto v_r,\\
                v_m&\mapsto v_m~\text{otherwise}.
            \end{align*} Moreover, $\sum\limits_{i=1}^{n}a_ig(v_i)=e_1$. Thus \[g \in G_n^p,\]and \[g\cdot[v_1,\dots,v_{n}]=-[v_1,\dots,v_{n}].\] So $[v_1,\dots,v_{n}]$ vanishes in $(\St_n(\Z)
            )_{G_n^p}$.
            \case $a_j\equiv a_\ell\equiv a_r\not\equiv 0\bmod p$ for some $1\leq j<\ell<r\leq n$. 
         
            Set \[w_j=v_j+v_\ell,\quad w_m=v_m\text{ for $m\neq j$},\]and define $g_1\in\SL_n(\Z)$ \begin{align*}
                w_\ell&\mapsto -w_\ell, \\
                w_r&\mapsto w_j,\\
                w_j&\mapsto w_r,\\
                w_m&\mapsto w_m~\text{otherwise}.
            \end{align*} Then $g_1\in G_n^p$ by \eqref{eq:e1}, and  \[g_1\cdot [v_1,\dots,v_j+v_{\ell},\dots,v_\ell,\dots,v_r,\dots,v_{n}]=- [v_1,\dots,v_j+v_{\ell},\dots,v_\ell,\dots,v_r,\dots,v_{n}].\] Similarly, there exists $g_2\in G_n^p$ that sends $[v_1,\dots,v_j,\dots,v_j+v_{\ell},\dots,v_r,\dots,v_{n}]$ to its negative. Thus, by \eqref{eq:v}, $[v_1,\dots,v_{n}]$ vanishes in $(\St_n(\Z))_{G_n^p}$.
            \case $a_j\equiv a_\ell\not\equiv 0 \bmod p$ and $a_r\equiv a_s\not\equiv 0 \bmod p$ for some $1\leq j<\ell<r<s\leq n$. 
            
            Set \[w_j=v_j+v_\ell, \quad w_m=v_m~\text{for $m\neq j$},\] and define $g_1 \in \SL_n(\Z)$ by\begin{align*}
                w_\ell&\mapsto -w_\ell, \\
                w_r&\mapsto w_s,\\
                w_s&\mapsto w_r,\\
                v_m&\mapsto v_m~\text{otherwise}.
            \end{align*} 
         Again, $g_1 \in G_n^p$ and
\[g_1 \cdot [v_1,\dots,v_j+v_\ell,\dots,v_{n}]
= -[v_1,\dots,v_j+v_\ell,\dots,v_{n}].\]
As in the previous case, a similar $g_2$ handles the second term from \eqref{eq:v}, and we conclude that $[v_1,\dots,v_{n}]$ vanishes in $(\St_n(\Z))_{G_n^p}$.
\end{mycases}

We have thus shown that the generators of $\St_n(\Z)$ vanish in the coinvariants $(\St_n(\Z))_{G_n^p}$. Therefore, $(\St_n(\Z))_{G_n^p}\cong 0$ as claimed.\end{proof}
\begin{remark} Since $G_n^p\subset\Gamma_{0,n}(p)$, it follows by \autoref{coinv K=0} that $\homology^{\binom{n}{2}}(\Gamma_{0,n}(p);\St_n(\Z))$ vanishes for $n\geq p+2$. This is weaker than the result of the first author \cite[Theorem A]{Abdelnaim}, which establishes vanishing for the range $n\geq \frac{p+14}{6}$ as well as for the primes $p\in\{2,3,5,7,13\}$.
   \end{remark}

        \begin{lemma}\label{K in Stab}
            Let $p$ be a prime, $n\geq 1$, and $k\geq 0$. Let $\vec A=(A_0,\dots,A_k)\in S_k$ be a decomposition $\Z^n=A_0\oplus\dots\oplus A_k.$ Then $G_{A_0}^p(v_0)\times\dots\times G_{A_k}^p(v_k)$ is a finite index subgroup of $\stab_{G_n^p}(\vec{A}).$
        \end{lemma}
        \begin{proof}
            Choose a representative $\vec A=(A_0,\dots,A_{k})\in S_k$ of the orbit $[\vec A]\in\Gamma_{0,n}(p)\backslash S_k$.  Then \[e_1\equiv v_0+\dots+v_k\bmod p\quad\text{for some $v_i\in A_i$}.\]Let $I=\{i\mid v_i\not\equiv 0\bmod pA_i\}$.
            Fix $g_i : A_i \to A_i$ for $i=0,\ldots,k$ and let \[g=g_0 \oplus \ldots \oplus g_k : \Z^n = A_0 \oplus \ldots \oplus A_k \longrightarrow A_0 \oplus \ldots \oplus A_k =\Z^n.\]  
Note that $g \in G_n^p$ if and only if \[\prod_{i=0}^k\det(g_i)=1,\] and \[g_i\cdot v_i\equiv v_i\bmod pA_i \quad\text{for all $i\in I$}.\] By construction, $g$ stabilizes $\vec A$ and all elements of $\GL_n(\Z)$ that stabilize $A$ are direct sums of this form. In other words,
\[
\stab_{G_n^p}(\vec{A})
=
\left\{
g_0 \oplus \dots \oplus g_k
\;\middle|\;
\begin{array}{c}
g_i\in\GL(A_i)\text{ for all }i,\\
\prod_{i=0}^{k}\det(g_i)=1,\\
g_i\cdot v_i\equiv  v_i \bmod pA_i
\text{ for all }i\in I
\end{array}
\right\}.
\]
        Consequently, $G_{A_0}^p(v_0)\times\dots\times G_{A_k}^p(v_k)$ is a subgroup of $\stab_{G_n^p}(\vec{A})$. Additionally, define $\delta\colon \stab_{G_n^p}(\vec{A})\longrightarrow \{\pm1\}^{k+1}$
by\[
\delta\left((g_0,\dots,g_k)\right)=\left(\det(g_0),\dots,\det(g_k)\right).\]
The image of $\delta$ is contained in the finite subgroup \[\left\{(\varepsilon_0,\dots,\varepsilon_k)\in\{\pm1\}^{k+1}\mid \varepsilon_0\cdots\varepsilon_k=1\right\}.\]
The kernel of $\delta$ is precisely $G_{A_0}^p(v_0)\times\dots\times G_{A_k}^p(v_k)$. Therefore, $G_{A_0}^p(v_0)\times\dots\times G_{A_k}^p(v_k)$ has finite index in $\stab_{G_n^p}(\vec{A})$.
        \end{proof}
\begin{proposition}\label{stab=0}
   Let $p$ be a prime, $k\geq 0$, and $n\geq (k+1)(p+1)+1$. Then \[E^1_{k0}(n)=0.\] 
   
\end{proposition}\begin{proof} 
Fix $n\geq (k+1)(p+1)+1$. Recall that \[E^1_{k0}(n)\cong \bigoplus_{[\vec{A}]\in G_n^p\backslash S_k}\homology_0\left(\stab_{G_n^p}(\vec{A});\St(A_0)\otimes\dots\otimes\St(A_k)\right),\]where $\vec A=(A_0,\dots,A_k)\in S_k$. We will show that each summand is zero. Since \[A_0\oplus\dots\oplus A_k=\Z^n,\] we have that $n=\sum\limits_{i=0}^k\rk(A_i)$. By \autoref{K in Stab}, $G_{A_0}^p(v_0)\times\dots\times G_{A_k}^p(v_k)$ is a subgroup of $\stab_{G_n^p}(\vec{A})$. Thus
passing to coinvariants induces a natural surjection
\begin{equation}\label{eq:surj}
\left(\St(A_0)\otimes\cdots\otimes\St(A_k)\right)_{G_{A_0}^p(v_0)\times\dots\times G_{A_k}^p(v_k)}
\rightarrow
\left(\St(A_0)\otimes\cdots\otimes\St(A_k)\right)_{\stab_{G_n^p}(\vec{A})}.\end{equation}
Therefore, it suffices to show that
\[\left(\St(A_0)\otimes\cdots\otimes\St(A_k)\right)_{G_{A_0}^p(v_0)\times\dots\times G_{A_k}^p(v_k)}
\cong0.\]

By the K\"unneth formula,
\begin{equation}\label{tensor}
\homology_0\left(G_{A_0}^p(v_0)\times\dots\times G_{A_k}^p(v_k);\St(A_0)\otimes\dots\otimes\St(A_k)\right)
\cong\bigotimes_{i=0}^{k}\homology_0\left(G_{A_i}^p(v_i);\St(A_i)\right).\end{equation}
Since
\[
n=\sum_{i=0}^k n_i\geq (k+1)(p+1)+1,
\]
there exists some $i$ such that $n_i\geq p+2$. If $v_i\neq 0\bmod pA_i$, then by \autoref{coinv K=0},
\[
\homology_0\left(G_{A_i}^p(v_i);\St(A_i)\right)\cong 0.
\]
If $v_i=0\bmod pA_i$, then
\[
\homology_0(G_{A_i}^p(v_i);\St(A_i))
=\homology_0(\SL(A_i);\St(A_i)) \cong 0\] for $n \geq 2$ by Lee-Szczarba \cite[Theorem 1.3]{LS}.
In either case, one tensor factor of \eqref{tensor} is zero. Thus
\[
\homology_0\left(G_{A_0}^p(v_0)\times\dots\times G_{A_k}^p(v_k);
\St(A_0)\otimes\dots\otimes\St(A_k)
\right)
\cong \left(\St(A_0)\otimes\dots\otimes\St(A_k)\right)_{G_{A_0}^p(v_0)\times\dots\times G_{A_k}^p(v_k)}\cong 0.\] It then follows by the surjection \eqref{eq:surj}\[
\homology_0\left(\stab_{G_n^p}(\vec{A});
\St(A_0)\otimes\dots\otimes\St(A_k)
\right)
\cong \left(\St(A_0)\otimes\dots\otimes\St(A_k)\right)_{\stab_{G_n^p}(\vec{A})}\cong 0.\qedhere\]
\end{proof}
\section{Vanishing of the codimension-1 cohomology of  \texorpdfstring{$\Gamma_{0,n}(p)$}{Lg}}
In this section, we prove \autoref{codim1}. The first step is to show that the differential \[f \colon E^1_{11}(n)\rightarrow E^{1}_{01}(n)\]is surjective.
\begin{lemma}\label{surj}
     Let $p$ be a prime. The map \[f \colon E^1_{11}(n)\rightarrow E^{1}_{01}(n)\]is surjective for $n\geq 3p+4$.
\end{lemma}
\begin{proof}
   Recall from \autoref{E-infty} that $E^\infty_{01}(n)=0$ for $n\geq 4$. Since the spectral sequence is first quadrant, there are no nonzero outgoing differentials from $E^r_{01}(n)$. The only possible nonzero incoming differential for $r \geq 2$ is
\[d^2\colon E^2_{20}(n)\to E^2_{01}(n).\] 
Since $n\geq 3p+4$, \autoref{stab=0} establishes that $E^1_{k0}(n)=0$ for $k\leq 2$. In particular
\[E^2_{20}(n)=0.\]
Therefore no nonzero higher differential can hit $E^2_{01}(n)$. Additionally, since no nonzero differential can leave $E^r_{01}(n)$, it follows that
\[E^2_{01}(n)=E^\infty_{01}(n)=0.\]
Thus
\[f=d^1 \colon E^1_{11}(n)\to E^1_{01}(n)
\]
is surjective.
\end{proof}

\begin{lemma}\label{surj'}
The natural map 
\begin{equation*}
    \begin{aligned}
\bigoplus_{k=1}^{n-1}\bigoplus_{i=0}^{k}
\bigoplus_{\substack{[(A_0,\dots,A_k)]\in G_n^p\backslash S_k\\
\rk (A_i) \le 3p+3}}
\homology_0(G_{A_0}^p(v_0);\St(A_0))\otimes& \cdots \otimes
\homology_1(G_{A_i}^p(v_i);\St(A_i))
\otimes \cdots \otimes
\homology_0(G_{A_k}^p(v_k);\St(A_k))\\&
\longrightarrow
\homology_1(G_n^p;\St_n(\Z))
\end{aligned}
\end{equation*} is surjective for $n \geq 3p+4$.
\end{lemma}
\begin{proof}
    By \autoref{surj}, the map \[f\colon E^1_{11}(n)\to E^1_{01}(n)=\homology_1(G_n^p;\St_n(\Z))\]is surjective for $n\geq 3p+4$. By \autoref{K in Stab}, for every $(A_0,A_k)\in S_1$, $G_{A_0}^p(v_0)\times G_{A_1}^p(v_1)$ is a finite index subgroup of $\stab_{G_n^p}((A_0,A_1))$. Since we work over $\Q$, \[G_{A_0}^p(v_0)\times G_{A_1}^p(v_1) \to \stab_{G_n^p}((A_0,A_1))\] induces a surjection on homology. Thus
\begin{equation}\label{twoterms}
\begin{aligned}
\bigoplus_{\substack{[(A_0,A_1)]\in G_n^p\backslash S_1\\ \rk(A_i)\leq n-1}}
\homology_0(G_{A_0}^p(v_0);\St(A_0))\otimes&
\homology_1(G_{A_1}^p(v_1);\St(A_1)) 
\quad\oplus\quad\homology_1(G_{A_0}^p(v_0);\St(A_0))
\otimes\homology_0(G_{A_1}^p(v_1);\St(A_1))\\&\qquad\qquad
\longrightarrow
\homology_1(G_n^p;\St_n(\Z))
\end{aligned}
\end{equation}is surjective for $n \geq 3p+4$.

    We will now prove 
    \begin{equation}\label{manyterms}
\begin{aligned}
\bigoplus_{k=1}^{n-1}\bigoplus_{i=0}^{k}
\bigoplus_{\substack{[(A_0,\dots,A_k)]\in G_n^p\backslash S_k\\
\rk (A_i) \le 3p+3}}
\homology_0(G_{A_0}^p(v_0);\St(A_0))\otimes& \cdots \otimes
\homology_1(G_{A_i}^p(v_i);\St(A_i))
\otimes \cdots \otimes
\homology_0(G_{A_k}^p(v_k);\St(A_k))\\&
\longrightarrow
\homology_1(G_n^p;\St_n(\Z))
\end{aligned}
\end{equation} is surjective by induction on $n$.

    For the base case, let $n=3p+4$. Then $n-1 =3p+3$. Thus, for $n=3p+4$, the map \eqref{twoterms} factors through \eqref{manyterms}. Since the map \eqref{twoterms} is surjective for all $n \geq 3p+4$, the map \eqref{manyterms} is surjective for $n=3p+4$. This establishes the base case of the induction.
 
Now fix $m \geq 3p+4$ and assume by induction that the map \eqref{manyterms} is surjective for all $n$ in the range $3p+4 \leq n <m$. The map \eqref{twoterms} for $n=m$ is surjective since $m \geq 3p+4$. Unlike the $3p+4$ case, there may be tensor factors of the form $\homology_1(G^p_{A_i}(v_i);\St(A_i))$ with $\rk(A_i) > 3p+3$ in the domain of the map from \eqref{twoterms}. 

If $v_i\equiv 0 \bmod pA_i$, then $\homology_1(G_{A_i}(v_i);\St(A_i))\cong 0$ by Church--Putman \cite[Theorem A]{CP}. If $v_i\not\equiv 0\bmod pA_i$, fix an isomorphism $$A_i \cong \Z^r$$ where $\rk(A_i)=r$. Since $r < m$, the induction hypothesis implies that the map \eqref{manyterms} is surjective for $n=r$. We can use this surjection to replace the terms of the form $\homology_1(G^p_{A_i}(v_i);\St(A_i))$ in the domain of the map \eqref{twoterms} with a sum of tensor products of several terms of the form \[\homology_0(G^p_{A_j}(v_j);\St(A_j))\] and exactly one term of the form \[\homology_1(G^p_{A_l}(v_l);\St(A_l))\] with $\rk(A_l) \leq 3p+3$. Thus, the map \eqref{manyterms} is surjective for $n=m$. The claim follows by induction.
 \end{proof}

   \begin{lemma}\label{lem:img}
       Let $A$ and $B$ be free abelian groups of ranks $m$ and $n$ respectively and let $h\geq 0$. Let $v_1\in A$ and $v_2\in B$. Then the maps \[\homology_h(G_A^p(v_1);\St(A))\otimes\homology_0(G_B^p(v_2);\St(B))\rightarrow\homology_h(G_{A\oplus B}^p(v_1+v_2);\St(A\oplus B))\]and \[\homology_0(G_B^p(v_2);\St(B))\otimes\homology_h(G_A^p(v_1);\St(A))\rightarrow\homology_h(G_{A\oplus B}^p(v_1+v_2);\St(A\oplus B))\]have the same image.
   \end{lemma}
   \begin{proof}
The first map is induced by the inclusion
\[G_A^p(v_1)\times G_B^p(v_2)\subseteq G_{A\oplus B}^p(v_1+v_2)
\]
together with the multiplication map
\[\mu_{A,B}\colon \St(A)\otimes \St(B)\longrightarrow \St(A\oplus B).\]
The second map is induced similarly using
\[G_B^p(v_2)\times G_A^p(v_1)\subseteq G_{B\oplus A}^p(v_2+v_1).\]
Let
\[
\tau\colon A\oplus B\longrightarrow B\oplus A,\qquad (a,b)\mapsto (b,a).
\]
Conjugation by $\tau$ identifies
$G_{A\oplus B}^p(v_1+v_2)
\cong G_{B\oplus A}^p(v_2+v_1),$
and $\tau$ also identifies $\St(A\oplus B)$ with $\St(B\oplus A)$. 

By the second relation of \autoref{Bykovski},
\[\tau\circ\mu_{A,B}=(-1)^{mn} \mu_{B,A}\circ s,\]where $s(x\otimes y)=y\otimes x.$ Thus the following diagram commutes up to multiplication by $(-1)^{mn}$:
\[
\begin{tikzcd}
\homology_h(G_A^p(v_1);\St(A))\otimes \homology_0(G_B^p(v_2);\St(B))
\arrow[r, "\alpha"]
\arrow[d, "s_*"']&
\homology_h(G_{A\oplus B}^p(v_1+v_2);\St(A\oplus B))
\arrow[d, "\tau_*"]
\\\homology_0(G_B^p(v_2);\St(B))\otimes \homology_h(G_A^p(v_1);\St(A))
\arrow[r,swap,"\beta"']&
\homology_h(G_{B\oplus A}^p(v_2+v_1);\St(B\oplus A))
\end{tikzcd}\] Since $s_*$ and $\tau_*$ are isomorphisms, the commutative diagram implies \[\tau_*(\Img(\alpha))=\Img(\beta).\] Since the second map in the statement is $\tau_*^{-1}\circ\beta$, it follows that
$\Img(\alpha)=\Img(\tau_*^{-1}\circ\beta).$
Therefore the two maps in the statement have the same image.
\end{proof}

    \begin{theorem}\label{H1=0}
        Let $p$ be a prime. Then \[\homology_1(G_n^p;\St_n(\Z))\cong 0\]for $n\geq 4p+5$.
    \end{theorem}
    \begin{proof} Fix $n \geq 4p+5$. By \autoref{surj'}, there is a surjective map \begin{equation}\label{map:s} \begin{aligned}
\bigoplus_{k=1}^{n-1}\bigoplus_{i=0}^{k}
\bigoplus_{\substack{[(A_0,\dots,A_k)]\in G_n^p\backslash S_k\\
\rk (A_i) \le 3p+3}}
\homology_0(G_{A_0}^p(v_0);\St(A_0))\otimes& \cdots \otimes
\homology_1(G_{A_i}^p(v_i);\St(A_i))
\otimes \cdots \otimes
\homology_0(G_{A_k}^p(v_k);\St(A_k))\\&
\longrightarrow
\homology_1(G_n^p;\St_n(\Z)).
\end{aligned}\end{equation}
By \autoref{lem:img}, the order of the tensor factors does not matter. Thus, the map \eqref{map:s} has the same image as the following map:
\begin{equation}\label{onei} \bigoplus_{k= 1}^{n-1} \bigoplus_{\substack{[(A_0,\dots,A_k)]\in G_n^p\backslash S_k\\  \rk (A_k) \le 3p+3}}
\homology_0(G_{A_0}^p(v_0);\St(A_0))\otimes\dots\otimes\homology_1(G_{A_k}^p(v_k);\St(A_k))\longrightarrow \homology_1(G_n^p;\St_n(\Z)).\end{equation} By grouping the $\homology_0$ tensor factors together, we see that the map \eqref{onei} factors through the following map:
\begin{equation}\label{vanishingdomeq}  \bigoplus_{\substack{[(A_0,A_1)]\in G_n^p\backslash S_1\\  \rk (A_1) \le 3p+3}}
\homology_0(G_{A_0}^p(v_0);\St(A_0))\otimes\homology_1(G_{A_1}^p(v_1);\St(A_1))\longrightarrow \homology_1(G_n^p;\St_n(\Z)).\end{equation} If $(A_0,A_1) \in S_1$ and $\rk (A_1) \leq 3p+3$, then \[\rk(A_0) \geq n-(3p+3) \geq (4p+5)-(3p+3)=p+2.\]
 By \autoref{coinv K=0}, $\homology_0(G_{A_0}^p(v_0);\St(A_0)) \cong 0$ if $\rk (A_0) \geq p+2$. Thus, the domain of the map of \eqref{vanishingdomeq} vanishes. Since the map from \eqref{vanishingdomeq} is surjective and zero, its codomain vanishes. 
\end{proof}

     We now conclude with the proof of \autoref{codim1}, which states that \[\homology_1(\Gamma_{0,n}(p);\St_n(\Z))= 0\]for $n\geq 4p+5$.
    \begin{proof}[Proof of \autoref{codim1}]
      We have that $G_n^p$ is a finite index subgroup of $\Gamma_{0,n}(p)$. Since we work over the rationals, we obtain a surjection
\[\homology_1(G_n^p;\St_n(\Z))
\rightarrow
\homology_1(\Gamma_{0,n}(p);\St_n(\Z)).\]
As the domain is zero by \autoref{H1=0} for $n\geq 4p+5$, it follows that $\homology_1(\Gamma_{0,n}(p);\St_n(\Z))$ is zero. By Borel--Serre duality, we conclude that
\[\homology^{\binom n2-1}(\Gamma_{0,n}(p);\mathbb Q)=0\]for $n\geq 4p+5$.
    \end{proof}
     \section{Vanishing of the codimension-2 cohomology of \texorpdfstring{$\Gamma_{0,n}(p)$}{Lg}}
     In this section, we prove \autoref{codim2}. Following the strategy of the previous section, we will show that the differential \[g \colon E^1_{12}(n)\rightarrow E^{1}_{02}(n)\]is surjective and zero in a stable range.
    \begin{lemma}\label{E1=0}
          Let $p$ be a prime, $k\geq 0$, and $n\geq 4p+5+k(p+1)$.
Then \[E^1_{k1}(n)=0.\]
    \end{lemma}
    \begin{proof}
        Recall that \[E^1_{k1}(n)\cong\bigoplus_{[\vec{A}]\in G_n^p\backslash S_k}\homology_1\left(\stab_{G_n^p}(\vec{A});\St(A_0)\otimes\dots\otimes\St(A_k)\right),\]where $\vec A=(A_0,\dots,A_k)\in S_k$. We will show that each summand is zero. By \autoref{K in Stab}, the product $G_{A_0}^p(v_0)\times\dots\times G_{A_k}^p(v_k)$ is a finite index subgroup of $\stab_{G_n^p}(\vec{A})$. Since we work over $\Q$, it is enough to show \[\homology_1\left(G_{A_0}^p(v_0)\times\dots\times G_{A_k}^p(v_k);\St(A_0)\otimes\dots\otimes\St(A_k)\right)\cong 0.\] By the K\"unneth formula, this group is isomorphic to  \[\bigoplus_{i=0}^k\left(\bigotimes_{j< i}\homology_0(G_{A_j}^p(v_j);\St(A_j))~\otimes~\homology_1(G_{A_i}^p(v_i);\St(A_i))~\otimes~\bigotimes_{j> i}\homology_0(G_{A_j}^p(v_j);\St(A_j))\right).\] 
        
        Fix $0\leq i\leq k$. If there exists $j\neq i$ such that $\rk(A_j)\geq p+2$, then \[\homology_0(G_{A_j}^p(v_j);\St(A_j))\cong 0\]
by \autoref{coinv K=0}. Thus the summand \[\bigotimes_{j< i}\homology_0(G_{A_j}^p(v_j);\St(A_j))~\otimes~\homology_1(G_{A_i}^p(v_i);\St(A_i))~\otimes~\bigotimes_{j> i}\homology_0(G_{A_j}^p(v_j);\St(A_j))\]vanishes in this case.

Therefore, a summand is zero unless $\rk(A_j)\leq p+1$ for all $j\neq i$. In this case, \[\sum_{j\neq i} \rk(A_j)\leq k(p+1).\] Since $n\geq 4p+5+k(p+1)$, we have\[\rk(A_i)=n-\sum_{j\neq i}\rk(A_j)\geq n-k(p+1)\geq 4p+5.\] If $v_i\not\equiv 0\bmod p$, then choosing an identification $A_i\cong \Z^{n_i}$ that sends $v_i$ to $e_1$ modulo $p$, together with \autoref{H1=0} gives
\[\homology_1(G_{A_i}^p(v_i);\St(A_i))\cong 0.\]
Now consider the case that $v_i\equiv 0\bmod p$. Then $G_{A_i}^p(v_i)=\SL(A_i)$. Since $\rk(A_i) \geq 4p+5 \geq 3$, 
\[\homology_1(G_{A_i}^p(v_i);\St(A_i))=\homology_1(\SL(A_i);\St(A_i))\cong 0\] by Church--Putman \cite[Theorem A]{CP}. Thus every summand  \[\bigotimes_{j< i}\homology_0(G_{A_j}^p(v_j);\St(A_j))~\otimes~\homology_1(G_{A_i}^p(v_i);\St(A_i))~\otimes~\bigotimes_{j> i}\homology_0(G_{A_j}^p(v_j);\St(A_j))\]vanishes. Consequently, \[\homology_1\left(G_{A_0}^p(v_0)\times\dots\times G_{A_k}^p(v_k);\St(A_0)\otimes\dots\St(A_k)\right)\cong 0.\]We conclude that $E^1_{k1}(n)=0$ for $n\geq 4p+5+k(p+1)$.
    \end{proof}
    We now turn to show that the map $g$ in the spectral sequence is surjective.
    \begin{lemma}
    \label{gursj}
          Let $p$ be a prime and $n\geq 6p+7$. Then the map \[g \colon E^1_{12}(n)\rightarrow E^{1}_{02}(n)\]is surjective.
    \end{lemma}
    \begin{proof}
       Recall by \autoref{E-infty} that $E^\infty_{02}(n)=0$ for $n\geq 4$. We will show that $E^2_{02}(n)=E^\infty_{02}(n)$.

        Since the spectral sequence is first quadrant, there are no nonzero outgoing nonzero differentials from $E^r_{02}(n)$. The only possible incoming nonzero higher differentials are \[d^2\colon E^2_{21}(n)\rightarrow E^2_{02}(n) \, \, \text{ and } \, \, d^3\colon E^3_{30}(n)\rightarrow E^3_{02}(n).\]
        \autoref{E1=0} establishes that $E^1_{21}(n)=0$ for $n\geq 4p+5+2(p+1)=6p+7$. In particular, \[E^2_{21}(n)=0\quad\text{for $n\geq 6p+7$}.\] Additionally, \autoref{stab=0} shows that $E^1_{30}(n)=0$ for $n\geq 4(p+1)+1=4p+5$. In particular, \[E^3_{30}=0\quad\text{for $n\geq 6p+7$}.\] Therefore \[
E^2_{02}(n)=E^\infty_{02}(n)=0.\]We conclude that \[g=d^1\colon E^1_{12}(n)\to E^1_{02}(n)\]is surjective for $n\geq 6p+7$.
    \end{proof}
    
    \begin{lemma}\label{surj:B}
         The natural map 
\[
\bigoplus_{k=1}^{n-1}
\bigoplus_{[(A_0,\ldots,A_k)]\in G_n^p\backslash S_k}
\Bigg(
\bigoplus_{\substack{0\le r\le k\\ \rk(A_r)\le 6p+6}}
\homology_0(G_{A_0}^p(v_0);\St(A_0))
\otimes\cdots\otimes\homology_2(G_{A_r}^p(v_r);\St(A_r))
\otimes\cdots\otimes\homology_0(G_{A_k}^p(v_k);\St(A_k))\]
\[\oplus\]\[\bigoplus_{\substack{0\le r<s\le k\\ \rk(A_r)\leq 3p+3\\ \rk(A_s)\le 3p+3}}
\homology_0(G_{A_0}^p(v_0);\St(A_0))\otimes\cdots\otimes
\homology_1(G_{A_r}^p(v_r);\St(A_r))
\otimes\cdots\otimes
\homology_1(G_{A_s}^p(v_s);\St(A_s))\otimes\cdots\otimes\homology_0(G_{A_k}^p(v_k);\St(A_k))
\Bigg)\]
\[\longrightarrow \homology_2(G_n^p;\St_n(\mathbb Z)).
\]is surjective for $n\geq 6p+7$.
    \end{lemma}
    \begin{proof}
        By \autoref{gursj}, the map \[g\colon E^1_{12}(n)\rightarrow E^1_{02}(n)=\homology_2(G_n^p;\St_n(\Z))\]is surjective for $n\geq 6p+7$. By \autoref{K in Stab}, $G_{A_0}^p(v_0)\times G_{A_1}^p(v_1)$ is a finite index subgroup of $\stab_{G_n^p}((A_0,A_1))$. Since we work over $\Q$, \[G_{A_0}^p(v_0)\times G_{A_1}^p(v_1) \to \stab_{G_n^p}((A_0,A_1))\] induces a surjection on homology. Thus 
        \begin{equation}\label{3terms}
           \begin{aligned}
\bigoplus_{\substack{[(A_0,A_1)]\in G_n^p\backslash S_1\\ \rk(A_i)\le n-1}}\Big(&\homology_0(G_{A_0}^p(v_0);\St(A_0))
 \otimes \homology_2(G_{A_1}^p(v_1);\St(A_1))\\&\oplus\;
\homology_2(G_{A_0}^p(v_0);\St(A_0))\otimes\homology_0(G_{A_1}^p(v_1);\St(A_1))\\&\oplus\;\homology_1(G_{A_0}^p(v_0);\St(A_0)) \otimes \homology_1(G_{A_1}^p(v_1);\St(A_1))
\Big)\longrightarrow\homology_2(G_n^p;\St_n(\Z))
\end{aligned}
        \end{equation}is surjective for $n\geq 6p+7.$ 

        We will now prove \begin{equation}\label{many}
  \begin{aligned}
\bigoplus_{k=1}^{n-1}
\bigoplus_{[(A_0,\ldots,A_k)]\in G_n^p\backslash S_k}
\Bigg(
\bigoplus_{\substack{0\le r\le k\\ \rk(A_r)\le 6p+6}}
\homology_0(G_{A_0}^p(v_0);\St(&A_0))
\otimes\cdots\otimes\homology_2(G_{A_r}^p(v_r);\St(A_r))
\otimes\cdots\otimes\homology_0(G_{A_k}^p(v_k);\St(A_k))\\
&\oplus\\\bigoplus_{\substack{0\le r<s\le k\\ \rk(A_r)\leq 3p+3\\ \rk(A_s)\le 3p+3}}
\homology_0(G_{A_0}^p(v_0);\St(A_0))\otimes\cdots\otimes
\homology_1(G_{A_r}^p(v_r);\St&(A_r)\otimes\cdots\otimes
\homology_1(G_{A_s}^p(v_s);\St(A_s))\otimes\cdots\otimes\homology_0(G_{A_k}^p(v_k);\St(A_k))
\Bigg)\\
&\hspace{-1cm}\longrightarrow \homology_2(G_n^p;\St_n(\mathbb Z)).
\end{aligned}
\end{equation}is surjective by induction on $n$.

  For the base case, let $n=6p+7$. We will show that the surjection map \eqref{3terms} factors through the map \eqref{many}. For the terms \[
\homology_0(G_{A_0}^p(v_0);\St(A_0))
\otimes\homology_2(G_{A_1}^p(v_1);\St(A_1))\quad\text{and}\quad
\homology_2(G_{A_0}^p(v_0);\St(A_0))
\otimes\homology_0(G_{A_1}^p(v_1);\St(A_1)),
\] we have both $\rk(A_0)$ and $\rk(A_1)$ are at most $6p+6$. Thus the $\homology_2$ tensor factor satisfies the rank condition appearing in \eqref{many}. 

\noindent It remains to consider the terms of the form \[\homology_1(G_{A_0}^p(v_0);\St(A_0))
\otimes
\homology_1(G_{A_1}^p(v_1);\St(A_1)).\] If both $\rk(A_0)$ and $\rk(A_1)$ are at most $3p+3$, then this term already appears in the second summand of \eqref{many}. Otherwise, one of the two ranks
is at least $3p+4$. Without loss of generality, suppose
$\rk(A_0)\ge 3p+4.$
Since $n=6p+7$, we then have
\[\rk(A_1)=n-\rk(A_0)\le 6p+7-(3p+4)=3p+3.
\] 
If $v_0 \equiv 0\bmod pA_0$, then Church--Putman \cite[Theorem A]{CP} implies \[\homology_1(G_{A_0}^p(v_0);\St(A_0)) =\homology_1(\SL(A_0);\St(A_0)) \cong 0\] since $\rk(A_0) \geq 3p+4  \geq 3$. Thus, 
we may assume that $v_0\not\equiv 0\bmod pA_0$. Fix an isomorphism $$A_0\cong\Z^{\rk(A_0)},$$ and use \eqref{surj'} to replace $\homology_1(G_{A_0}^p(v_0);\St(A_0))$ with a sum of product of several terms of the form $\homology_0(G^p_{A_j}(v_j);\St(A_j))$ and a term of the form $\homology_1(G^p_{A_l}(v_l);\St(A_l))$ with $\rk(A_l) \leq 3p+3$. Repeating this for the other factor if necessary, the second summand of \eqref{3terms} factors through the second summand of \eqref{many}. 

\noindent Thus the map \eqref{3terms} factors through the map \eqref{many}. Since the map \eqref{3terms} is surjective for $n\geq 6p+7$, the map \eqref{many} is surjective for $n=6p+7$. This establishes the base case.

Now fix $m \geq 6p+7$ and assume by induction that the map \eqref{many} is surjective for all $n$ in the range $6p+7 \leq n <m$. The map \eqref{3terms} for $n=m$ is surjective since $m \geq 6p+7$. Unlike the $6p+7$ case, there may be tensor factors of the form $\homology_2(G^p_{A_i}(v_i);\St(A_i))$ with $\rk(A_i) > 6p+6$ in the domain of the map \eqref{3terms}. 

First consider a term containing an $\homology_2$ tensor factor, say $
\homology_2(G_{A_0}^p(v_0);\St(A_0)).$ If $\rk(A_0)\leq 6p+6$, then this term already appears in the map
\eqref{many} so we can assume $\rk(A_0)>6p+6$. If $v_0\equiv 0\bmod pA_0$, then \[
\homology_2(G_{A_0}^p(v_0);\St(A_0))
=\homology_2(\SL(A_0);\St(A_0)) \cong 0\] for $n \geq 4$ by Brück, Patzt, Sroka, Wilson and the second author \cite[Theorem B]{BMPSW}. Now consider the case that $v_0\not\equiv 0\bmod pA_0$. Fix an isomoprhism $$A_0\cong \Z^{\rk(A_0)}.$$ Since $\rk(A_0)<m$, the induction hypothesis applies to $A_0$. It follows that we may replace 
$\homology_2(G_{A_0}^p(v_0);\St(A_0))$
by a sum of tensor products of the two types appearing in
\eqref{many}:
\begin{enumerate}
    \item tensor products with one $\homology_2$ tensor factor, supported
    on a summand of rank at most $6p+6$;
    \item tensor products with two $\homology_1$ tensor factors, supported
    on summands of ranks at most $3p+3$.
\end{enumerate} Thus, tensoring with the $\homology_0$ tensor factor shows that the original term factors through the domain of \eqref{many}.

It remains to consider terms of the form
\[
\homology_1(G_{A_0}^p(v_0);\St(A_0))
\otimes
\homology_1(G_{A_1}^p(v_1);\St(A_1)).
\]
If both $\rk(A_0)$ and $\rk(A_1)$ are at most $3p+3$, then this term
already appears in \eqref{many}. If one of the ranks is greater than
$3p+3$, we apply \autoref{surj'} to that $\homology_1$ tensor factor. Repeating
this for the other factor if necessary, we replace the term
$\homology_1(G_{A_0}^p(v_0);\St(A_0))
\otimes
\homology_1(G_{A_1}^p(v_1);\St(A_1))$ by a sum of tensor products of terms of the form \[\homology_0(G_{A_j}^p(v_j);\St(A_j))\] and exactly two terms of the form \[\homology_1(G_{A_{l_1}}^p(v_{l_1});\St(A_{l_1})\quad\text{and}\quad \homology_1(G_{A_{l_2}}^p(v_{l_2});\St(A_{l_2})\] with $\rk(A_{l_1}),\rk(A_{l_2})\leq 3p+3$. Thus these terms factor
through the second summand in the domain of \eqref{many}.

Therefore the surjection map \eqref{3terms} factors through the map \eqref{many}. Since \eqref{3terms} is surjective, the map \eqref{many} is surjective for $n=m$. The claim follows by induction. \end{proof}
    \begin{theorem}\label{H2=0}
          Let $p$ be a prime. Then \[\homology_2(G_n^p;\St_n(\Z))\cong 0\]for $n\geq 7p+8$.
    \end{theorem}
    \begin{proof}
        Fix $n\geq 7p+8$. By \autoref{surj:B}, there is a surjective map \begin{equation}\label{map:sB}
\begin{aligned}
\bigoplus_{k=1}^{n-1}
\bigoplus_{[(A_0,\ldots,A_k)]\in G_n^p\backslash S_k}
\Bigg(
\bigoplus_{\substack{0\le r\le k\\ \rk(A_r)\le 6p+6}}
\homology_0(G_{A_0}^p(v_0);\St(&A_0))
\otimes\cdots\otimes\homology_2(G_{A_r}^p(v_r);\St(A_r))
\otimes\cdots\otimes\homology_0(G_{A_k}^p(v_k);\St(A_k))\\
&\oplus\\\bigoplus_{\substack{0\le r<s\le k\\ \rk(A_r)\leq 3p+3\\ \rk(A_s)\le 3p+3}}
\homology_0(G_{A_0}^p(v_0);\St(A_0))\otimes\cdots\otimes
\homology_1(G_{A_r}^p(v_r);\St&(A_r)\otimes\cdots\otimes
\homology_1(G_{A_s}^p(v_s);\St(A_s))\otimes\cdots\otimes\homology_0(G_{A_k}^p(v_k);\St(A_k))
\Bigg)\\
&\hspace{-1cm}\longrightarrow \homology_2(G_n^p;\St_n(\mathbb Z)).
\end{aligned}
\end{equation}is surjective for $n\geq 6p+7$. By \autoref{lem:img}, this map has the same image as the following map:

\begin{equation}\label{twoi}
\begin{aligned}
\bigoplus_{k=1}^{n-1}\Biggl(&
\bigoplus_{\substack{[(A_0,\dots,A_k)]\in G_n^p\backslash S_k\\
\rk(A_k)\leq 6p+6}}
\homology_0(G_{A_0}^p(v_0);\St(A_0))\otimes\dots\otimes
\homology_2(G_{A_k}^p(v_k);\St(A_k)) \\
\oplus\;&
\bigoplus_{\substack{[(A_0,\dots,A_k)]\in G_n^p\backslash S_k\\
\rk(A_{k-1})\leq 3p+3\\\rk(A_k)\leq 3p+3}}
\homology_0(G_{A_0}^p(v_0);\St(A_0))\otimes\dots\otimes
\homology_1(G_{A_{k-1}}^p(v_{k-1});\St(A_{k-1}))\otimes
\homology_1(G_{A_k}^p(v_k);\St(A_k))
\Biggr)\\&\hspace{5cm}\longrightarrow \homology_2(G_n^p;\St_n(\Z)).
\end{aligned}
\end{equation}Note that the second summand necessarily contains at least one $\homology_0$ tensor factor, since $n\geq 7p+8$. 

\noindent By grouping the $\homology_0$ tensor factors together, we get that the map \eqref{twoi} factors through the following map:
\begin{equation}\label{zerodomain}
\begin{aligned}
&\bigoplus_{\substack{[(A_0,A_1)]\in G_n^p\backslash S_1\\
\rk(A_1)\leq 6p+6}}
\homology_0(G_{A_0}^p(v_0);\St(A_0))\otimes
\homology_2(G_{A_1}^p(v_1);\St(A_1))\\\oplus\;&
\bigoplus_{\substack{[(A_0,A_1,A_2)]\in G_n^p\backslash S_2\\
\rk(A_{1})\leq 3p+3\\\rk(A_2)\leq 3p+3}}
\homology_0(G_{A_0}^p(v_0);\St(A_0))\otimes
\homology_1(G_{A_{1}}^p(v_{1});\St(A_{1}))\otimes
\homology_1(G_{A_2}^p(v_2);\St(A_2))\\&\hspace{5cm}
\longrightarrow \homology_2(G_n^p;\St_n(\Z)).
\end{aligned}
\end{equation} 
 If $(A_0,A_1) \in S_1$ and $\rk (A_1) \leq 6p+6$, then \[\rk(A_0) \geq n-(6p+6) \geq (7p+8)-(6p+6)=p+2.\]
If $(A_0,A_1,A_2) \in S_2$ and $\rk (A_1),\rk(A_2) \leq 3p+3$, then \[\rk(A_0) \geq n-(3p+3) \geq (7p+8)-(3p+3)=4p+5>p+2.\]
 By \autoref{coinv K=0}, $\homology_0(G_{A_0}^p(v_0);\St(A_0)) \cong 0$ if $\rk (A_0) \geq p+2$. Thus, the domain of the map from \eqref{zerodomain} vanishes. Since the map from \eqref{zerodomain} is surjective and zero, its codomain vanishes. 
    \end{proof}
    Lastly we prove \autoref{codim2}, which states that \[\homology_2(\Gamma_{0,n}(p);\St_n(\Z))\cong 0\]for $n\geq 7p+8$.
    \begin{proof}[Proof of \autoref{codim2}]
  Since $G_n^p$ is a finite index subgroup of $\Gamma_{0,n}(p)$, we obtain a surjection
\[\homology_2(G_n^p;\St_n(\Z))\longrightarrow\homology_2(\Gamma_{0,n}(p);\St_n(\Z)).\]As the domain is zero by \autoref{H2=0} for $n\geq 7p+8$, it follows that $\homology_2(\Gamma_{0,n}(p);\St_n(\Z))$ is zero. We thus conclude by Borel--Serre duality that \[\homology^{\binom{n}{2}-2}(\Gamma_{0,n}(p);\Q)\cong 0\]for $n\geq 7p+8$.
    \end{proof}
\bibliographystyle{alpha}
\bibliography{References}
\end{document}